\documentclass[12pt,leqno]{article}
\usepackage{amsfonts}

\usepackage{amsmath}
\usepackage{a4,latexsym,amssymb,amsfonts}
\usepackage{enumerate}
\usepackage{amsthm}

\numberwithin{equation}{section} \topmargin -1cm

\newtheorem{theo}{Theorem}[section]
\newtheorem{lemma}[theo]{Lemma}
\newtheorem{prop}[theo]{Proposition}

\newtheorem{defi}[theo]{Definition}

\theoremstyle{definition}
\newtheorem{rem}[theo]{Remark}

\newtheorem{ass}[theo]{Assumptions}

\def\clip{C_{Lip}^1}
\def\vertv{\vert_{_{\V}}}
\def\t{\tau}
\def\a{\alpha}
\def\A{A_0}
\def\V{V^\prime}
\def\ranglev{\rangle_{V^\prime}}
\def\lipd{ ]\kern-1pt_{_{L}}}
\def\lips{[}
\def\H{\mathcal{H}}
\def\ro3{\theta}
\def\e{\varepsilon}

\def\B{\mathcal{B}}
\def\qed{\hbox{\hskip 6pt\vrule width6pt height7pt
               depth1pt  \hskip1pt}\bigskip}

\begin{document}

\title{Dynamic programming for infinite
horizon boundary control problems of PDE's with age structure}
\author{Silvia Faggian\footnote{LUM ``Jean Monnet", Casamassima, I-70010, faggian@lum.it},
Fausto Gozzi\footnote{LUISS ``Guido Carli", Roma, I-00162,
fgozzi@luiss.it}} \maketitle

\begin{abstract}
We develop the dynamic programming approach for a
family of infinite horizon boundary control problems with linear
state equation and convex cost. We prove that the value function
of the problem is the unique regular solution of the associated
stationary Hamilton--Jacobi--Bellman equation and use this to
prove existence and uniqueness of feedback controls. The idea of
studying this kind of problem comes from economic applications, in
particular from models of optimal investment with vintage capital.
Such family of problems has already been studied in the finite
horizon case in \cite{Fa2}\cite{Fa3}. The infinite horizon case is
more difficult to treat and it is more interesting from the point
of view of economic applications, where what mainly matters is the
behavior of optimal trajectories and controls in the long run. The
study of infinite horizon is here performed through a nontrivial
limiting procedure from the corresponding finite horizon problem.
\end{abstract}
{\footnotesize {\bf Keywords.} Linear convex control, boundary
control, Hamilton--Jacobi--Bellman equations, age-structured
systems, optimal control, economic growth, vintage models.}

\bigskip

\section{Introduction\label{INTRO}}
This paper is devoted to the study of a family of infinite horizon
boundary control problems with linear state equation and convex
cost and of the associated stationary Hamilton--Jacobi--Bellman
(briefly, HJB) equations by means of Dynamic Programming. More
precisely, we let $H$ and $U$ be separable real Hilbert spaces
with scalar products $(\cdot\vert\cdot)_H$ and
$(\cdot\vert\cdot)_U$ respectively, and we consider a dynamical
system of the following type
\begin{equation}\label{eq stato in H}\begin{cases}y^\prime(\t)=\A y(\t)+Bu(\t), & \t\in
]t,+\infty[\\
y(t)=x\in H,\end{cases}\end{equation} where $H$ is the state space,
$y:[t,+\infty[\to H$ is the trajectory, $U$ is the control space and
$u:[t,+\infty[\to U$ is the control, $A_0:D(A_0)\subset H\to H$ is
the infinitesimal generator of a strongly continuous semigroup of
linear operators $\{e^{\t\A}\}_{\t\ge0}$ on $H$,
 and the control operator $B$ is linear and {\it unbounded}, say
 $B:U\to[D(A_0^*)]^\prime$.
Besides,  we consider an {\it infinite horizon} cost functional
given by
\begin{equation}\label{J in H}J_\infty(t,x,u)
=\int_t^{+\infty}e^{-\lambda\tau}\left[g_0\left(y(\t)\right)+h_0\left(u(\t)\right)\right]d\t\end{equation}
where
the function $g_0$ is convex and $C^1$, and $h_0$ is {\it l.s.c.},
convex, superlinear, possibly infinite valued, as better specified
later. Our problem is that of minimizing $J_\infty(t,x,u)$ with
respect to $u$ over a set $\mathcal{U}$ of admissible controls
(which will be denoted with $L^p_\lambda(t,+\infty;U)$, that is an
$L^p$ space with a suitable weight, as defined in Section
\ref{mainsec}).

Then the value function is defined as
\begin{equation}Z_\infty(t,x)=\sup_{u \in L^p_\lambda(t,+\infty;U) }J_\infty(t,x,u).\end{equation}
Since it is easily shown that $Z_\infty(t,x)=e^{-\lambda t}Z_\infty(0,x)$, it is enough to study the HJB
equation associated to the problem with initial time $t=0$, that is
\begin{equation}\label{HJintro}-\lambda \psi(x)+(
\psi^\prime(x)\; |\; \A x)_H-h_0^*(-B^*\psi^\prime(x))+g(x)=0,\
x\in H\end{equation}  whose candidate solution is $Z_\infty(0,x)$.
(Here and in the sequel, $h_0^*$ indicates the L\'egendre
transform of the convex {\it l.s.c.} function $h_0$.)

The problem has been already studied by Faggian and by Faggian and
Gozzi in the papers \cite{Fa2,Fa3,FaSC,FaGo} in the case of finite
horizon, with and without constraints on the control and on the
state, yielding a definition of generalized solutions of the
associated evolutionary HJB equation. This paper studies instead
the infinite horizon case.

It is well known that linear convex control problems with
unbounded control operator $B$ in Hilbert spaces arise when one
rephrases in abstract terms some boundary control problem for
linear PDEs (or, more generally, problems with control on
subdomains). Indeed, we motivate our framework with the
application to a problem of optimal investment with vintage
capital arising in economic theory (originally formulated by
Barucci and Gozzi \cite{BG1,BG2} and then studied also in
\cite{F1,F2,F3}), that cannot be treated with the existing results
and that we describe in detail in Section \ref{es eco}. We also
observe that our framework adapts also to other optimal control
problems driven by first order PDE's or by delay equations and
arising in models of population dynamics (see e.g.
\cite{Almeder,F1}), advertising (see e.g.
\cite{FaGo,FHS,GM1,Marinelli}), general equilibrium with vintage
capital (see e.g. \cite{Boucekkine,FabGoz}).

Our main results are stated in Section \ref{mainsec},  in Theorems
\ref{main}, \ref{th:mainnew}, \ref{th:uniquefeedback}, where we
prove that the value function $ Z_{\infty}(0, \cdot )$ is the
unique regular ($C^1$) solution of the HJB equation
(\ref{HJintro}) and that there exists a unique optimal control
strategy in feedback form. Moreover the value function is the
limit of value functions of suitable finite horizon problems.

We obtain the results by means of the procedure introduced by
Barbu and Da Prato \cite{BD1} (see also Di Blasio \cite{D1,D2})
that consists, roughly speaking, in the following steps:
\begin{itemize}
    \item  Consider a family of suitable problems with finite horizon $T$, with
     value functions $\Psi_T$
     and  show they are  the unique
    regular solutions of the corresponding family of evolutionary HJB equations.\footnote{Indeed
    these facts in our case were shown in \cite{Fa2,Fa3},
    refining the convex regularization method
    by Barbu and Da Prato contained in \cite{BD1}.}
    \item Show that the value functions $\Psi_T$  converge, as $T\to +\infty$,
    to a regular function
    $\Psi_\infty$.
    \item Prove that $\Psi_\infty$ is the unique solution of the stationary
    HJB equation and that it is equal to the value function, $Z_\infty$, of the
    infinite horizon problem with initial time $t=0$; prove the existence
    and uniqueness of optimal feedbacks.
\end{itemize}
In our (boundary control) case a sharp refinement of this methods
is needed. Indeed, with respect to the  papers quoted above, our
problem features two new nontrivial difficulties:

\begin{itemize}
    \item The presence of the boundary control yields the unboundedness of the
control operator $B$ in the state equation (\ref{eq stato in H})
and, as a consequence, the discontinuity of the Hamiltonian in the
HJB equation (\ref{HJintro}). This fact, coupled with the
non-analyticity of the semigroup generated by $A$, induces us to
work in an enlarged space $\V\supset H$. This setting was already
 introduced in \cite{Fa2,Fa3} to treat the corresponding finite horizon problem.
 Of course, since in the examples in Section \ref{es eco} the parameters
  have significance only in $H$,
  we need to prove that when in the extended setting the initial datum $x$ is in $H$,
  then the whole optimal trajectory lies in $H$, and the optimal control behaves accordingly.
    \item The running costs $g_0$ and $h_0$ are not
    bounded from
    below. This means that a two-sided inequality has to be proved in
    order to show the convergence as $T \to +\infty$. To this extent, we exploit the
     coercivity
    of the function $h_0$ to derive that optimal controls are bounded in $L^p_\lambda$.
\end{itemize}

We may say that the contribution of this paper is both to
mathematics and to applications: the results contained in the
theorems in Section \ref{mainsec} extend the existing theory of
regular solutions of HJB equations in Hilbert spaces to a new set of
problems; on the other hand our results are the basis to start the
study of the properties of the optimal state-control pairs in the
economic problem, and in those other applications that can can be
framed into the same setting.

The paper is organized as follows. In next the rest of this Section
(Subsection \ref{TheLiterature}) we review the main literature on
HJB equations in Hilbert spaces. In Section \ref{preliminari} we
recall the definition of strong solution and the results on
existence and uniqueness of strong solutions in the finite horizon
case, as they appear in \cite{Fa2}. In Section \ref{mainsec} we
present the problem in the extended space and we state the main
results. Section \ref{PROOFS} is devoted to proofs and  Section
\ref{es eco} to the description of the economic example.

\subsection{The Literature}
\label{TheLiterature} We recall that optimal control problems for
infinite dimensional systems and the associated HJB equation have
been studied in two different frameworks: one is that
 of classical and strong solutions, and
the other is that of viscosity solutions. We recall also that, as
far as we know, verification techniques have been performed in
infinite dimension just in the classical/strong context, for they
require the value function to be regular (at least in the state
variable).

Regarding Dynamic Programming for boundary control problems only few
results are available. For the case of linear systems and quadratic
costs (where HJB equation reduces to the operator Riccati equation)
the reader is referred {\it e.g.} to the book by Lasiecka and
Triggiani \cite{LT}, to the book by Bensoussan, Da Prato, Delfour
and Mitter \cite{BDDM}, and, for the case of nonautonomous systems,
to the papers by Acquistapace, Flandoli and Terreni \cite{AFT, AT1,
AT2, AT3}. For the case of a linear system and a general convex
cost, we mention the papers by Faggian \cite{ Fa1, Fa1bis, Fa2, Fa3,
FaSC}, by Faggian and Gozzi \cite{FaGo}. On Pontryagin maximum
principle for boundary control problems we mention again the book by
Barbu and Precupanu (Chapter 4 in \cite{BP}).

For the case of distributed control the literature is indeed richer:
we refer the reader to Barbu and Da Prato \cite{BD1, BD2,BD3} for
some linear convex problems, to Di Blasio \cite{D1,D2} for the case
of constrained control, to Cannarsa and Di Blasio \cite{CD} for the
case of state constraints, to Barbu, Da Prato and Popa \cite{BDP}
and to Gozzi \cite{G1,G2,G3} for semilinear systems.

For viscosity solutions and HJB equations in infinite dimension we
mention the series of papers by Crandall and Lions \cite{CL} where
also some boundary control problem arises. Moreover, for boundary
control we mention the papers by Cannarsa, Gozzi and Soner
\cite{CGS} and by Cannarsa and Tessitore \cite{CT} on existence and
uniqueness of viscosity solutions of HJB equation.  We note also
that a verification theorem in the case of viscosity solutions has
been proved in some finite dimensional case in \cite{GSZ,YZ}.

Regarding applications, in addition to the economic literature
recalled above, we refer the reader to the many examples contained
in the books by Lasiecka and Triggiani \cite{LT} and by Bensoussan
{\it et al} \cite{BDDM}.

\bigskip
\section{Preliminaries: the finite horizon case.}\label{preliminari}
We here recall all the relevant
results on the \emph{finite horizon} case that are needed in the
sequel. According to the notation in \cite{Fa2}, if $X$ and $Y$
are Banach spaces, we set
\begin{equation*}\begin{split}
&Lip(X;Y)=\{f:X\to Y ~:~\lips f\lipd:=\sup_{x,y\in X,~x\neq y}
\frac{\vert f(x)-f(y)\vert_{Y}}{\vert x-y\vert_X} <+\infty\}\\
&\clip(X):=\{f\in C^1(X)~:~ \lips f^\prime\lipd<+\infty\}\\
&\B_p(X,Y):=\{f:X\to\mathbb{R}~:~\vert f\vert_{\B_p}:=\sup_{x\in
X} {\vert f(x)\vert_Y\over 1+\vert x\vert_X^p}<+\infty\},\ \ \
\B_p(X):=\B_p(X,\mathbb{R}).\\
\end{split}\end{equation*}
Moreover we set
\begin{equation*}
\Sigma_0(X):=\{w\in \B_2(X)\ :\ w\ {\rm is\ convex,\ } w\in\clip(X)
\}\end{equation*}and, for $T>0$
\begin{equation*}\begin{split}\mathcal{Y}([0,T]\times X)=
\{w:[0,T]&\times X\to \mathbb{R}\ :\ w\in C([0,T],\B_2(X)),\
\\w(t,\cdot)\in&\Sigma_0(X),\ \forall t\in[0,T], \ \ w_x\in C([0,T], \B_1(X,{X^\prime}))\}\\
\end{split}
\end{equation*}

\bigskip

\noindent Then we consider two Hilbert spaces $V,\V$, being dual
spaces, which we do not identify for reasons which are recalled in
Remark \ref{noidentif} and we indicate with
$\langle\cdot,\cdot\rangle$ the duality pairing. We set $\V$ as
the state space of the problem, and denote with $U$ the control
space, being $U$ another Hilbert space.

Given an initial time $t \ge 0$ an initial state $x\in \V$, a finite
horizon $T>t$, a number $p>1$, and a control $u \in L^p(t,T;U)$ we
consider the trajectory in $\V$
\begin{equation}\label{mild}y(\t)=e^{(\t-t)A}x+\int_t^\t
e^{(\t-\sigma)A}Bu(\sigma)d\sigma,\ \ \tau\in[t,T] ,\
\end{equation}
and a profit functional of type
\begin{equation}\label{J in V}J_T(t,x,u)
=\int_t^T\left[g\left(\t,y(\t)\right)+h\left(\t,u(\t)\right)\right]d\t+
\varphi(y(T)).\end{equation} We deal with the problem of
minimizing $J_T(t,x,\cdot)$ over all $u \in L^p(t,T;U)$ taking the
following set of assumptions on the data.
\begin{ass}\label{asst}

\begin{enumerate}

\item[1.] $A:D(A)\subset\V\to\V$ is the infinitesimal generator of
a strongly continuous semigroup $\{e^{\t A}\}_{\t\ge0}$ on $\V$;

\item[2.] $B\in L(U,\V)$; \item[3.] there exists $\omega_0 \ge 0$
such that $\vert e^{\t A}x\vertv\le e^{\omega_0 \t}\vert
x\vertv,~\forall \t\ge0$; \item[4.] $g\in \mathcal{Y}([0,T]\times
\V)$, $t\mapsto\lips g_x(t,\cdot)\lipd\in L^1(0,T)$
\item[5.]$\varphi\in\Sigma_0(\V)$; \item[6.] $h(t,\cdot)$ is
convex, lower semi--continuous, $\partial_u h(t,\cdot)$  is
injective for all $t\in[0,T]$. \item[7.] If is set
$\H(t,u):=[h(\t,\cdot)]^*(u)$, then we assume $\H\in
\mathcal{Y}([0,T]\times U)$, $\H(t,0)=0$, and
$\sup_{t\in[0,T]}\lips \H_u(t,\cdot)\lipd<+\infty.$

\end{enumerate}\end{ass}

\begin{rem}\label{noidentif} We do not identify $V$ and $\V$ for
 in the applications the problem is naturally set in a Hilbert space $H$, such that
$V\subset H\equiv H^\prime\subset\V$ (with all bounded inclusions).
Indeed, in order to avoid the discontinuities due to the presence of $B$,
as they appear in (\ref{eq stato in H})(\ref{J in H}), we work in
the extended state space $\V$ related to $H$ in the following way: $V$ is the Hilbert space
$D(\A^*)$ endowed with the scalar product $(v|w)_V:=
(v|w)_H+(\A^*v|\A^*w)_H$, $\V$ is the dual space of $V$ endowed
with the operator norm. Then assume that $B\in L(U,\V)$, and extend
 the semigroup $\{e^{tA_0}\}_{t\ge0}$ on $H$ to a semigroup
$\{e^{tA}\}_{t\ge0}$ on the space ${V^\prime}$, having infinitesimal
generator $A$, a proper extension of $A_0$. The reader is referred
to \cite{Fa3} for  a detailed treatment. \hfill\qed \end{rem}

\begin{rem} \label{estensioni1} Note that the functions
$g$ and $\phi$ arising from applications usually appear to be
defined and $C^1$ on $H$, not on the larger space $\V$. Then, we
here need to {\it assume} that they can be extended to {\it
$C^1$-regular functions on $\V$} - which is a non trivial issue. We
refer the reader to Section \ref{es eco} to see how such extension
is obtained in the specific case of the economic example, and to
\cite{Fa2} and \cite{Fa3} for a thorough discussion of this
issue.\hfill\qed\end{rem}

\begin{rem}\label{asssuh} In Assumption \ref{asst}[7], we assumed $\H(t,0)=0$. Such assmption
is not restrictive since $\H(t,0)=- \inf_{v\in U}h(t,v)$ and, if
this value is not $0$, we may reduce to this case simply setting
$\bar g=g+\inf_{v\in U}h(t,v)$ and $\bar h =h-\inf_{v\in U}h(t,v)$
and treating the problem with $\bar g$ and $\bar h$ in place of $g$
and $h$.  Note also that the assumption $\partial h(t,\cdot)$
injective is intended to yield a good definition for $\H_u$ as it
is, roughly speaking,  $\H_u=(\partial h)^{-1}$. Note also that once
one has the datum $h$, its convex conjugate $\H$ is very often
explicitly computed. Then the assumptions on $\H$ are essentially
assumptions on its convex conjugate $h$, but more conveniently
stated to ensure $\H$ has the desired properties. \hfill\qed
\end{rem}

\medskip

Such optimal control problem can be associated by means of dynamic
programming, to the following Hamilton-Jacobi-Bellman equation
\begin{equation}\label{HJBb}
\begin{cases}v_{t}(t,x)-\H(t,-B^*v_{x}(t,x))+\langle
  A x\vert  v_{x}(t,x)\rangle+g(t,x)=0, &(t,x)\in[0,T]\times {V^\prime}\\
v(T,x)=\varphi(x),&\\
\end{cases}\end{equation}
that can be written, by the change if variable $v(t,x)=\phi(T-t,x)$,
as
 \begin{equation}\label{HJBf}\begin{cases}\phi_t(t,x)+\H(T-t,-B^*\phi_x(t,x))-\langle
  Ax,\phi_x(t,x)\rangle=g(T-t,x), &(t,x)\in[0,T]\times {V^\prime}\\
\phi(0,x)=\varphi(x).&\\
\end{cases}\end{equation}

Finally, the value function of the problem is defined as
\begin{equation}\label{value}W_T(t,x)=\inf_{u\in L^p(t,T;U)} J_T(t,x,u),\end{equation}

Indeed in \cite{Fa2} Faggian proved existence and uniqueness of
strong solutions, as defined shortly afterwards, for a class of more
general HJB equations, that is

\begin{equation}\label{HJBgen}\begin{cases}\phi_t(t,x)+F(t,\phi_x(t,x))-\langle
  Ax,\phi_x(t,x)\rangle=g(T-t,x), &(t,x)\in[0,T]\times {V^\prime}\\
\phi(0,x)=\varphi(x),&\\
\end{cases}\end{equation}
where $F$ satisfies
\begin{equation}\label{HJBFH} F\in \mathcal{Y}([0,T]\times V),\ \  F(t,0)=0,\ \
\sup_{t\in [0,T]}[ F_p(t,\cdot)]_L<+\infty\end{equation}
Note indeed that if we set
 \begin{equation*}
 F(t,p):=\H(T-t,-B^*p)=\sup_{u\in U}\{( u\vert -B^*p)_U-h(T-t,u)\}.\end{equation*}
then $F$ satisfies (\ref{HJBFH}) and it  is  well defined for $p$ in $V$, to which
 $\phi_x(t,x)$ belongs.

\begin{defi} \label{strongdefi} Let  Assumptions \ref{asst} $1-5$, and $(\ref{HJBFH})$ be satisfied. We say
that   $\phi\in C([0,T],\B_2(\V))$ is a {\rm strong} solution of
$(\ref{HJBgen})$ if there exists a family $\{\phi^\e\}_\e\subset
C([0,T],\B_2(\V))$ such that:

$(i)$ $\phi^\e(t,\cdot)\in \clip(\V)$ and $\phi^\e(t,\cdot)$ is
convex for all $t\in[0,T]$; $\phi^\e(0,x)=\varphi(x)$ for all
$x\in\V$.

$(ii)$ there exist  constants $\Gamma_1,\Gamma_2>0$ such that
$$\sup_{t\in[0,T]}\lips\phi_x^\e(t)\lipd\le \Gamma_1,~
\sup_{t\in[0,T]}\vert\phi_x^\e(t,0)\vert_V\le \Gamma_2,~\forall
\e>0;$$

$(iii)$  for all $x\in D(A)$, $t\mapsto\phi^\e(t,x)$ is continuously
differentiable;

$(iv)$ $\phi^\e\to\phi$, as $\e\to 0+$,
 in $C([0,T],\B_2(\V))$;

$(v)$ there exists  $g_\e\in C([0,T];\B_2(\V))$ such that, for all
$t\in[0,T]$ and $x\in D(A)$,
$$\phi_t^\e(t,x)-F(t,\phi_x^\e(t,x))+\langle Ax,\phi_x^\e(t,x)\rangle=g_\e(T-t,x)$$
with $g_\e(t,x)\to g(t,x)$, and $\int_0^T\vert
g_\e(s)-g(s)\vert_{C_2}ds\to 0$, as $\e\to 0+.$ \end{defi}

The main result contained in \cite{Fa2} is the following.

\begin{theo}\label{exun}
Let Assumptions \ref{asst} $1-5$, and $(\ref{HJBgen})$ be satisfied. There exists a unique
strong solution $\phi$ of $(\ref{HJBf})$ in the class
$C([0,T],\B_2(\V))$ with the following properties:

$(i)$ for all $x\in D(A)$, $\phi(\cdot,x)$ is Lipschitz
continuous;

$(ii)$ $\phi\in\mathcal{Y}([0,T]\times \V)$. Moreover the
following estimate is satisfied for all $t\in[0,T]$
\begin{equation}\label{R(t)}[\phi_x(t)]_L\le
e^{2\omega_0 t} [ \varphi^\prime]_L + \int_0^t e^{2\omega_0 (t-s)}
[ g_x(T-s,\cdot)]_L  ds.\end{equation}

\end{theo}

Regarding applications to the optimal control problem, in
\cite{Fa3} we were able to prove what follows.

\begin{theo}\label{verif}
Let Assumptions \ref{asst} $1-7$ be satisfied, and let $\phi$ be the strong
  solution of $(\ref{HJBf})$ described in Theorem \ref{exun}. Then
  $$W_T(t,x)=\phi(T-t,x),~\forall t\in[0,T],~\forall x\in\V,$$
  that is,
  the value function $W_T$ of the optimal control problem is the
  unique strong solution of the backward HJB equation $(\ref{HJBb})$.
\end{theo}

\bigskip

\section{The infinite horizon problem}\label{mainsec}
We describe the abstract setup of the infinite
horizon optimal control problem and we state the main result of
the paper, namely Theorem \ref{th:mainnew}, that establishes that the
value function of our problem is the unique regular solution of
the associated HJB equation. Some other important
results follow, such as Theorem \ref{th:uniquefeedback}
on existence and uniqueness of optimal
feedbacks, and Theorem \ref{main}, establishing the connection
between finite and infite horizon value functions. Proofs of all assertions are found in section \ref{4}.

We use the same framework as that  in Section \ref{preliminari},
for the finite horizon problem. As one expects, the state
space is $\V$ and the control space is $U$. The state equation is
given in $\V$ as
\begin{equation}y(\t)=e^{(\t-t)A}x+\int_t^\t
e^{(\t-\sigma)A}Bu(\sigma)d\sigma,\ \ \tau\in[t,+\infty[ ,\
\end{equation}
while,
for all $x\in\V$ and $t>0$, the target functional is of type
\begin{equation}\label{J-t-T}J_\infty(t,x,u):=\int_t^{+\infty}e^{-\lambda
\tau}[g_0(y(\tau))+h_0(u(\tau))]d\tau.\end{equation}
We assume the following hypotheses:

\begin{ass}\label{asst2}
\begin{enumerate}

\item[1.] $A:D(A)\subset\V\to\V$ is the infinitesimal generator of
a strongly continuous semigroup $\{e^{\t A}\}_{\t\ge0}$ on $\V$;

\item[2.] $B\in L(U,\V)$;

\item[3.] there exists $\omega\in \mathbb{R}$ such that $\vert
e^{\t A}x\vertv\le  e^{\omega \t}\vert x\vertv,~\forall \t\ge0$;

\item[4.]  $g_0, \phi_0\in\Sigma_0(\V)$

\item[5.] $h_0$ is  convex, lower semi--continuous,
$\partial_u h_0$  is  injective.

\item[6.] $h_0^*(0)=0$, $h_0^*\in \Sigma_0(V)$;
\item[7.] $\exists a>0$,  $\exists b\in \mathbb{R}$, $\exists p>1$ : $h_0(u)\ge a\vert u\vert_U^p+b$,
$\forall u\in U$;

Moreover, either
\item[8.a] $p> 2$, $\lambda>(2\omega\vee\omega)$.

 or
 \item[8.b] $\lambda>\omega$,  and $g_0,\phi_0 \in \B_1(\V).$
\end{enumerate}\end{ass}
\begin{rem}Note that the Assumption \ref{asst2} [3]
above implies that also Assumption \ref{asst} [3] where
$\omega_0=\omega\vee0$.\end{rem} The functional $J_\infty(t;x,u)$
has to be minimized with respect to $u$ over the set of admissible
controls
\begin{equation}\label{admiss} L^p_\lambda(t,+\infty;U)=\{u\in L^1_{loc}(t,+\infty;U)
\ ;\ t\mapsto u(t)e^{-\frac{\lambda t}{p}}\in L^p(t,+\infty;U)\},
\end{equation}
which is Banach space with the norm $$
\|u\|_{L^p_\lambda(t,+\infty;U)}=\int_t^{+\infty} |u(\tau)|_U^p
e^{-\lambda \tau}d\tau = \|e^{-\frac{\lambda (\cdot)
}{p}}u\|_{L^p(t,+\infty;U)}.
$$
Similarly, the space $L^p_\lambda(t,s;U)$ endowed with the norm
$$
\|u\|_{L^p_\lambda(t,s;U)}=\int_t^{s} |u(\tau)|_U^p e^{-\lambda
\tau}d\tau = \|e^{-\frac{\lambda (\cdot) }{p}}u\|_{L^p(t,s;U)}.
$$
is a Banach space. Then (\ref{admiss}) is the natural set of admissible controls
 to get estimates in this setting (see e.g Lemma
\ref{stima y} and Lemma \ref{OC sono limitati}).

 The value function is then defined as
$$Z_\infty(t,x)=\inf_{u\in L^p_\lambda(t,+\infty;U)}J_\infty(t,x,u).$$  As it is easy to
check that
$$Z_\infty(t,x)=e^{-\lambda t}Z_\infty(0,x)$$
one may associate to the problem the following stationary HJB
equation
\begin{equation}\label{SHJB}-\lambda \psi(x)+\langle
\psi^\prime(x), A
x\rangle-h_0^*(-B^*\psi^\prime(x))+g(x)=0,\end{equation} whose
candidate solution is the function $Z_\infty(0,\cdot)$.

We will use the following definition of solution for equation
(\ref{SHJB}).

\begin{defi}\label{defsolSHJB}
A function $\psi$ is a \textrm{classical} solution of the
stationary HJB equation $(\ref{SHJB})$ if it belongs to $\Sigma_0
(\V)$ and satisfies $(\ref{SHJB})$ pointwise for every $x\in D(A)$.
\end{defi}

\begin{rem} \label{estensioni}  The reader has
certainly realized that Assumptions \ref{asst2} $[1-7]$ imply
 Assumptions \ref{asst} $[1-7]$.
 Moreover, as mentioned thoroughly  in Remark \ref{estensioni1}, we
need to assume that the functions $g_0$ and $\phi_0$  can be
extended to $C^1$-regular functions on $\V$.\hfill\qed\end{rem}

\begin{rem} \label{strettaconvessita} See Remark \ref{asssuh} for some
comments on $h_0$ and $h_0^*$ that apply also to this
case.\hfill\qed\end{rem}

Before proving that the value
function of the infinite horizon problem starting at $(0,x)$,
namely $Z_\infty(0,x)$, is the unique classical solution to the
stationary HJB equation, some preliminary work is needed.
First we show that $Z_\infty(0,x)$ is the limit as $t$ tends to
$+\infty$ of a suitable family of value functions for finite
horizon, along with their gradients. Doing so, we also
establish that $Z_\infty$ inherits from that family the $C^1$
regularity in $x$ which we need to solve the stationary HJB
equation, and which is so precious when building optimal feedback
maps.

\begin{theo}\label{main}
Let Assumptions \ref{asst2} be satisfied.
Let also $\phi_T(t,x)$
be the unique strong solution to $(\ref{HJBf})$. Then the function
$$\Psi(t,x):=e^{\lambda(T-t)}\phi_T(t,x)$$ is independent of $T$
and there exists the following limit
$$\Psi_\infty(x):=\lim_{t\to+\infty}\Psi(t,x).$$
The convergence is uniform on bounded subsets of $\V$. Moreover, if
$\lambda>\omega\max\{2, \frac{p}{p-1}\}$, then $\Psi_\infty \in
\Sigma_0 (\V)$. Moreover, for every fixed $x\in \V$
$$\Psi_x(t,x)\to\Psi^\prime_\infty(x), \ {\rm weakly\ in\ }V,\  {\rm as\ \ }t\to+\infty.$$
 \end{theo}

Hence, $\Psi_\infty$ being  the candidate solution to the
stationary HJB equation (\ref{SHJB}), one shows what follows.

\begin{theo}\label{th:mainnew} Let Assumptions \ref{asst2} hold. Then:

$(i)$ $\Psi_\infty$ is the value function of the infinite
horizon problem with initial time $t=0$, that is
$$\Psi_\infty(x)=Z_\infty(0,x)=\inf_{u\in L^p_\lambda(0,+\infty;U)}J_\infty(0,x,u).$$
Moreover $Z_\infty(t,x)=e^{-\lambda t}\Psi_\infty(x)$;

$(ii)$  $\Psi_\infty$ is a classical solution (as
defined in Definition \ref{defsolSHJB}) of the stationary
Hamilton-Jacobi-Bellman equation $(\ref{SHJB})$. that is
$$-\lambda \Psi_{\infty}(x)+\langle \Psi_{\infty}^\prime(x), Ax\rangle-h_0^*(-B^*
\Psi_{\infty}^\prime(x))+g(x)=0.$$

$(iii)$ The  function $\Psi_\infty$ is the unique classical solution
to $(\ref{SHJB})$.
\end{theo}

Once we have established that $\Psi_\infty$ is the classical
solution to the stationary HJB equation, and that it is
differentiable, we can build optimal feedbacks and prove the
following theorem.

\begin{theo}\label{th:uniquefeedback}
Let Assumptions \ref{asst2} hold. Let $t\ge 0$ and $x\in \V$ be
fixed. Then there exists a unique optimal pair $(u^*,y^*)$.
The optimal state $y^*$ is the unique solution of the Closed Loop
Equation
\begin{equation}\label{CLE} y(\t)=e^{(\t-t)A}x+\int_t^\t
e^{(\t-\sigma)A}B(h_0^*)'(-B^*\Psi^\prime_\infty(y(s))) d\sigma,\
\ \tau\in[t,+\infty[ .\
\end{equation}
while the optimal control $u^*$ is given by the feedback formula
$$
u^*(s) = (h_0^*)'(-B^*\Psi^\prime_\infty(y^*(s))).
$$
\end{theo}

\section{Proofs of the main results}\label{4}
\label{PROOFS} In this section we prove the theorems stated in
Section \ref{mainsec}.

\subsection{Auxiliary functions, equations and estimates}
We study infinite horizon by means of finite horizon.
Then it is worth noting that, thanks to the particular dependence
of data on the time variable, we can associate to the HJB equation  arising
in finite horizon the following equation:
\begin{equation}\label{hjb2}
\begin{cases}z_t(t,x)-\lambda z(t,x)+\langle Ax,z_x(t,x)\rangle
- h_0^*(-B^*z_x(t,x))+g_0(x)=0&
\\
z(T,x)=\phi_0(x)\end{cases}\end{equation}
and define a strong solution of (\ref{hjb2}) as follows.
\begin{defi} Let $(t,x)\in[0,T]\times {V^\prime}$.
We say that $Z_T$ is a \textrm{strong solution} to
 $(\ref{hjb2})$ if
$$Z_T(t,x)=e^{\lambda t}v_T(t,x)$$
with $v_T$ any strong solution to $(\ref{HJBb})$, in the sense of
Definition \ref{strongdefi}.\end{defi}

\begin{rem}
Equation (\ref{hjb2}) is obtained formally from
(\ref{HJBb}) with the change of variable $v(t,x)=e^{-\lambda t}
z(t,x)$. Note that one could give a direct definition of
solution of (\ref{hjb2}) (without passing through strong solutions
of (\ref{HJBb})) in the spirit of Definition $(\ref{strongdefi})$.
\hfill\qed\end{rem}

\noindent Note that
$$u\in L^p(t,T;U)\iff u\in L_\lambda^p(t,T;U)$$
so that all minimization procedure in Section 2 can be
equivalently operated  in $L^p(t,T;U)$ or in $L_\lambda^p(t,T;U)$.
Then, recalling that the unique strong solution to $(\ref{HJBb})$
is the value function of the optimal control problem (see
\cite{Fa3}), the following result is readily proven.
\begin{prop} Let Assumptions \ref{asst} be satisfied, and let
$Z_T$ be the unique strong solution to $(\ref{hjb2})$. Then
\begin{equation}\label{Z}Z_T(t,x)=\inf_{u\in
L^p_\lambda(t,T;U)}\left\{\int_t^{T}e^{-\lambda
(\tau-t)}[g_0(y(\tau))+h_0(u(\tau))]d\tau+e^{-\lambda(T-
t)}\phi_0(y(T))\right\}.\end{equation} \end{prop} We may also
write a forward version of $(\ref{hjb2})$, that is
\begin{equation}\label{hjb3}\begin{cases}
\psi_t(t,x)+\lambda \psi(t,x)-\langle Ax,\psi_x(t,x)\rangle
+ h_0^*(-B^*\psi_x(t,x))=g_0(x)&\\
\psi(0,x)=\phi_0(x)\end{cases}\end{equation} with
$(t,x)\in[0,T]\times {V^\prime}$ and $\psi(t,x)=z(T-t,x)$, where $Z$ is the
unique strong solution of $(\ref{hjb2})$, and then prove the
following important result.
\begin{lemma}Let $\Psi_T(t,x)=Z_T(T-t,x)$, where $Z$ is given by $(\ref{Z})$. Then
$\Psi$ does not depend on $T$, that is
\begin{equation}\label{Psi}\Psi(t,x)\equiv\Psi_T(t,x)=\inf_{u\in L^p_\lambda(0,t;U)}\left\{\int_0^{t}e^{-\lambda
\tau}[g_0(y(\tau))+h_0(u(\tau))]d\tau+e^{-\lambda
t}\phi_0(y(t))\right\}.\end{equation} Hence, . Moreover a Dynamic Programming Principle holds
\begin{equation}\label{PPDT}\Psi(t,x)=\kern-15pt\inf_{u\in L^p_\lambda(0,s;U)}\kern-10pt\left\{\int_0^{s}e^{-\lambda
\tau}[g_0(y(\tau))+h_0(u(\tau))]d\tau+e^{-\lambda
s}\Psi(t-s,y(s))\right\},\ \forall s\in[0, t].\end{equation}
\end{lemma}

\begin{proof} By changing the variable, for all $0<s<t$, and $u\in L^p_\lambda(s,t;U)$, we have
\begin{equation}\label{funzio}
J_{t}(s,x,u(\cdot))=e^{-\lambda
  s}J_{t-s}(0,x,u(\cdot+s)).\end{equation}
Then by definition of value function and
$(\ref{Z})$, we have
$$\Psi_T(t,x)=\inf_{u\in L^p_\lambda(T-t,T;U)}
J_{T}(T-t, x, u)=\inf_{\bar u\in L^p_\lambda(0,t;U)}J_t(0,x,\bar
u)$$
%
%
%
where the last equality is obtained by setting
$\bar u(s):=u(s+T-t)$, and observing that
$$u\in L^p_\lambda(T-t,T;U)\iff \bar u\in L^p_\lambda(0,t;U).$$
Then (\ref{Psi}) follows by generality of $\bar u$.

The proof of the Dynamic Programming Principle is standard and we omit it.
\end{proof}

Here follow  some other technical results that will be frequently exploited in the sequel.

\begin{lemma}\label{stima y} In Assumptions \ref{asst2}, if $u\in L^p_\lambda(t,T;U)$ is an admissible control and
$y(\tau)\equiv y(\tau;t,x,u)$ is the associated trajectory, and
$q=p/(p-1)$, then for suitable positive  constant $C$ independent
of $t$ and $x$ the following estimates hold:

\begin{equation}\label{y-stima-1,7}\int_t^se^{-\omega\tau}\vert
u(\tau)\vert_{U}d\tau\le {\ro3(t,s)}^{\frac{1}{q}}\Vert
u\Vert_{L_\lambda^p(t,s;U)}\end{equation}
\begin{equation}\label{y-stima-1}\vert y(\tau)\vert_{\V}\le
Ce^{\omega\tau}\big[\vert
x\vertv+\ro3(t,\tau)^{\frac{1}{q}}\Vert
u\Vert_{L_\lambda^p(t,\tau;U)}\big]
\end{equation}
\begin{equation}\label{muovastima}\int_t^se^{-\lambda\tau}\vert y(\tau)\vertv d\tau\le
C \left[\vert x\vertv+\Vert u\Vert_{L_\lambda^p(t,s;U)}
+1\right],
\ \forall s\ge t\end{equation}
with
$$\ro3(t,s)=\begin{cases} \frac{p-1}{\vert\lambda-p\omega\vert}\vert
e^{q\left(\frac{\lambda}{p}-\omega\right)t}-
e^{q\left(\frac{\lambda}{p}-\omega\right)s}\vert &\lambda\not=\omega p\\
\vert t-s\vert &\lambda=\omega p.\end{cases}$$
\end{lemma}

\begin{rem}\label{omeganeg} Note that all inequalities in the following proof hold also for $p=2$, but the main results in section \ref{mainsec} require also $p>2$.
Moreover, in the set of Assumptions \ref{asst2}
$$\omega<0\implies\lambda>\omega p.$$
Indeed, from [8.b] follows $p>1$ which implies $\{\lambda \ :\ \omega<\lambda<p\omega\}=\emptyset$, while
from [8.a] follows $p\ge2$ which implies $\{\lambda \ :\ 2\omega<\lambda<p\omega\}=\emptyset$. \end{rem}

\begin{proof} In what follows we denote by $C$ a positive constant not depending on $t$, $x$ and $u$.
 Inequality (\ref{y-stima-1,7}) holds by means
H\"older's inequality. From (\ref{y-stima-1,7}) follows
\begin{equation}\label{nuovastima}\begin{split}\vert y(\tau)\vert_{\V}
&\le \max\{\Vert
B\Vert_{L(U,\V)},e^{\vert\omega\vert t}\}
e^{\omega\tau} \bigg(\vert
x\vert_{\V}+\int_t^{\tau} e^{-\omega\sigma}\vert
u(\sigma)\vert_Ud\sigma\bigg)\\
&\le Ce^{\omega\tau}\bigg[\vert
x\vertv+\ro3(t,\tau)^{\frac{1}{q}}\Vert
u\Vert_{L_\lambda^p(t,\tau;U)}\bigg],\end{split}\end{equation} so
that also (\ref{y-stima-1}) is proven. To prove
(\ref{muovastima}) we need to estimate the right hand side in
\begin{equation}\label{boh}\begin{split}
\int_t^s e^{-\lambda\tau}\vert y(\tau)\vertv&\le \frac{C}{\lambda}(e^{-\lambda t}-e^{-\lambda s}) \vert
x\vertv+C
\Vert u\Vert_{L_\lambda^p(t,s;U)}
\int_t^s e^{-(\lambda-\omega)\tau}\ro3(t,\tau)^{\frac{1}{q}}d\tau
\end{split}\end{equation}
Indeed, in case
$\lambda>\omega p$, one derives
\begin{equation}\label{rho3}\begin{split}e^{-(\lambda-\omega)\tau}\ro3(t,\tau)^{\frac{1}{q}}&\le
Ce^{-(\lambda-\omega)\tau}\left[e^{q\left(\frac{\lambda}{p}-\omega\right)\tau}-
e^{q\left(\frac{\lambda}{p}-\omega\right)t}\right]^{\frac{1}{q}}\\
&\le C e^{-\frac{\lambda}{q}\tau},\\
\end{split}\end{equation}
while similarly in case
$\lambda<\omega p$, one obtains
$$e^{-(\lambda-\omega)\tau}\ro3(t,\tau)^{\frac{1}{q}}\le C e^{-(\lambda-\omega)\tau}e^{\left(\frac{\lambda}{p}-\omega\right)t}\le C e^{-(\lambda-\omega)\tau}$$
(we recall that in such case $\lambda-\omega>0$ in view of Remark \ref{omeganeg}).
Hence, when $\lambda\not=\omega p$, we have
\begin{equation}\label{omegapnot=lambda}\begin{split}\int_t^s e^{-(\lambda-\omega)\tau}\ro3(t,\tau)^{\frac{1}{q}}d\tau
&\le C \Vert u\Vert_{L_\lambda^p(t,s;U)}\
e^{-\left[\frac{\lambda}{q}\wedge(\lambda-\omega)\right]t},\end{split}\end{equation}
for all $s$. In the case
$\lambda=\omega p$, on the other hand, there exists $\delta>0$ such
that $\lambda-\omega-\delta>0$, and consequently
 $T_\delta\ge t$
such that
$$e^{-\delta\tau}\vert \tau-t\vert^{\frac{1}{q}}\le1,\ \forall\tau\ge T_\delta.$$
Then
 \begin{equation}\label{rho31}\begin{split}\int_t^s&e^{-(\lambda-\omega)\tau}\ro3(t,\tau)^{\frac{1}{q}}d\tau\le
\int_t^{+\infty}e^{-(\lambda-\omega)\tau}\ro3(t,\tau)^{\frac{1}{q}}d\tau\\
&\le\vert
T_\delta-t\vert^{\frac{1}{q}}\int_t^{{T_\delta}}e^{-(\lambda-\omega)\tau}
d\tau+\int_{T_\delta}^{+\infty}e^{-(\lambda-\omega-\delta)\tau}d\tau
 \\
 &\le\frac{1}{\lambda-\omega}\vert
 T_\delta-t\vert^{\frac{1}{q}}(e^{-(\lambda-\omega)
t}-e^{-(\lambda-\omega)
T_\delta})+\frac{1}{\lambda-\omega-\delta}e^{-(\lambda-\omega-\delta)T_\delta}
 \\
 &\le C
\end{split}\end{equation}
for a suitable constant $C$. Then applying (\ref{rho31}) and
(\ref{omegapnot=lambda}) to  (\ref{boh}) one derives (\ref{muovastima}).
\end{proof}

\begin{lemma}\label{OC sono limitati} Let Assumptions \ref{asst2} be satisfied.
Let $\e\in[0,1]$ be fixed, $u_\e\in L^p_\lambda(t,s;U)$ be any
$\e$-optimal control at $(t,x)$, with horizon $s$ for the
functional $J_s(t,x,\cdot)$ defined in $(\ref{J-t-T})$. Then, for
a suitable positive constant $K$, independent of $t$, $s$ and $x$,
we have:

$(i)$ $\Vert u_\e\Vert_{L_\lambda^p(t,s;U)} \le K(1+\vert
x\vertv^2)$, when Assumptions \ref{asst2}[8.a] holds;

$(ii)$ $\Vert u_\e\Vert_{L_\lambda^p(t,s;U)} \le K(1+\vert
x\vertv)$, when Assumptions \ref{asst2}[8.b] holds.
\end{lemma}

\begin{proof} Let $\bar u\in dom(h_0)$, and $\bar u(\tau)\equiv \bar u$.
Let also $(u_\e, y_\e)$ be $\e$-optimal at $(t,x)$. Then
\begin{equation}\label{fond}
J_s(t,x,u_\e)-\e\le J_s(t,x,\bar u).\end{equation}
On one hand, from the convexity of $g_0$ and $\phi_0$, and from (\ref{muovastima}),
there exists some positive constant $C_0$ such that
\begin{equation}\begin{split}
J_s(t,x,u_\e)&\ge \int_t^se^{-\lambda\tau}(a\vert u_\e(\tau)\vert^p_U+b) d\tau-
C_0\int_t^se^{-\lambda \tau}(1+\vert y_\e(\tau)\vertv) d\tau+\\
&\ \ \ \ \ \ \ \ \ \ \ \ \ \ \ \ \ \ \ \ \ \ \ -
C_0e^{-\lambda s}
(1+\vert y_\e(s)\vertv )\\
&\ge a\Vert
u_\e\Vert^p_{L_\lambda^p(t,s;U)}
 +\frac{b-C_0}{\lambda}
-C(\vert x\vertv+\Vert u_\e\Vert_{L_\lambda^p(t,s;U)}+1)+\\
&\ \ \ \ \ \ -C_0e^{-\lambda s}- \ CC_0
e^{-(\lambda-\omega)s}(\vert x \vert+\Vert
u_\e\Vert_{L_\lambda^p(t,s;U)} {\ro3(t,s)}^{\frac{1}{q}})\\
\end{split}\end{equation}
where in the last estimate we applied the assumptions on $h_0$,
and estimates (\ref{y-stima-1}) (\ref{muovastima}). Since
$e^{-(\lambda-\omega)s}{\ro3(t,s)}^{\frac{1}{q}}$ is bounded for all $s$,
the latter implies
\begin{equation}\label{fond2}
J_s(t,x,u_\e)\ge a\Vert u_\e\Vert^p_{L_\lambda^p(t,s;U)}-\gamma_1
\Vert u_\e\Vert_{L_\lambda^p(t,s;U)}-\gamma_2\vert
x\vertv+\gamma_3\end{equation} for a suitable choice of the
constants $\gamma_1, \gamma_2, \gamma_3$.
 On the other hand, also
$J_s(t,x,\bar u)$ can be estimated by means of either
 Assumptions \ref{asst2}[8.a] or [8.b]. Indeed,
 we derive that the trajectory $\bar
y(\tau)=y(t,x,\bar u)$
 satisfies
 $$\vert \bar y(\tau) \vert_{\V}\le
 K_1 e^{\omega\tau}(1+\vert x\vert_{\V}),$$
where $K_1=e^{-\omega t}(1\vee\Vert B\Vert \vert\bar u\vert
\omega^{-1})$. Then, if [8.a] holds, $\vert \bar y(\tau)
\vert_{\V}\le 2K_1^2 e^{2\omega\tau}(1+\vert x\vert^2_{\V})$, so
that
\begin{equation}\label{fond3}\begin{split}
J_s(t,x,\bar u)&\le (\vert h_0(\bar u)\vert+\vert
g_0\vert_{\B_2})\lambda^{-1}+2\vert g_0\vert_{\B_2} K_1^2(1+\vert
x\vert^2_{\V})
(\lambda-2\omega)^{-1}+\\&+\vert\phi_0\vert_{\B_2}e^{-\lambda
s}(1+ 2K_1^2 e^{2\omega\tau}(1+\vert x\vert^2_{\V}))\\
&\le \gamma_4 (1+\vert x\vert^2_{\V})
\end{split}\end{equation}
for a suitable constant $\gamma_4$. Hence, by means of
(\ref{fond}), (\ref{fond2}) and (\ref{fond3}), we obtain
$$\Vert u_\e\Vert^p_{L_\lambda^p(t,s;U)}(a\Vert u_\e\Vert^p_{L_\lambda^p(t,s;U)}-\gamma_1)
\le (2\gamma_2+\gamma_4) (1+\vert x\vert^2)-\gamma_3+\e$$ which
imply $(i)$. If instead [8.b] holds, then by a similar reasoning
one derives
\begin{equation}\label{fond4}
\begin{split}J_s(t,x,\bar u)&\
\le \gamma_5 (1+\vert x\vert_{\V})
\end{split}\end{equation}
for a suitable constant $\gamma_5$, and then $(ii)$.
\end{proof}

\begin{lemma} \label{stima psi} Let Assumptions \ref{asst2} be satisfied. If Assumptions \ref{asst2} hold with [8.a]), then $\Psi$ satisfies
$$\exists C>0\ \ :\ \ \vert\Psi(t,x)\vert\le C(1+\vert x\vert^2_{\V}),\ \forall (t,x)\in[0,+\infty[\times\V.$$
 If  Assumptions \ref{asst2} hold with [8.b]), then $\Psi$ satisfies
$$\exists C>0\ \ :\ \ \vert\Psi(t,x)\vert\le C(1+\vert x\vert_{\V}),\ \forall (t,x)\in[0,+\infty[\times\V.$$
\end{lemma}
\begin{proof} Let $\e>0$ be fixed and $u_\e$ be an admissible control,
with $y_\e(\tau)=y(\tau;0,x,u_\e)$ the associated trajectory, such
that
$$\Psi(t,x)\ge\int_0^t
e^{-\lambda\tau}\big[g_0(y_\e(\tau))+h_0(u_\e(\tau))\big]d\tau+e^{-\lambda
t }\phi_0(y_\e(t))-\e.$$ Hence from the convexity of $g_0$ and
$h_0$, for a suitable positive constant $\gamma$, and by applying
$(\ref{muovastima})$  we derive
\begin{equation*}\begin{split}
\Psi(t,x)&\ge -\gamma\int_0^te^{-\lambda\tau}\bigg[1+\vert
y_\e(\tau)\vert_{\V}+\vert u_\e(\tau)\vert_U\bigg]d\tau-\e\\
&\ge-\frac{\gamma}{\lambda}-\gamma C\big[\vert x \vertv+\Vert
u_\varepsilon\Vert_{L_\lambda^p(0,t;U)}
+1\big]-\gamma \vert 1-e^{-\lambda
t}\vert^{\frac{1}{q}}\lambda^{-\frac{1}{q}}
\Vert u_\varepsilon\Vert_{L_\lambda^p(0,t;U)}-\e\\
\end{split}\end{equation*}  so that by means of
 Lemma \ref{OC sono limitati}
$$-\Psi(t,x)\le C(1+\vert x\vert^2_{\V})$$
when Assumptions \ref{asst2}[8.a] holds, and
$$-\Psi(t,x)\le C(1+\vert x\vert_{\V}),$$
when Assumptions \ref{asst2}[8.b] holds, for a suitable choice of
the constant $C$. The missing inequality derives from
$$\Psi(t,x)\le J_t(0,x,\bar u)$$
with $\bar u(\tau)=\bar u\in dom(h_0)$, when we apply
(\ref{fond3}) when [8.a] holds, or
(\ref{fond4}) when [8.b] holds.
\end{proof}

\begin{lemma}\label{stimegrad}
Let Assumptions \ref{asst2} hold, and $\lambda>\omega\max\{2,q\}$.
Then
\bigskip

 $(i)$\ \ \
$\displaystyle\sup_{t\ge0}[\Psi_x(t)]_L<+\infty$;


\bigskip
$(ii)$ \
$\displaystyle\sup_{t\ge0}\vert\Psi_x(t,0)\vert_V<+\infty.$
\end{lemma}

\begin{proof} We use estimate (\ref{R(t)}) with $g(s,x)=e^{-\lambda s}g_0(x)$, and
$\varphi(x)=e^{-\lambda T}\phi_0(x)$  to derive
$$[\phi_x(t)]_L\le  e^{2\omega t}e^{-\lambda
T}\left([\phi_0^\prime]_L+\frac{[g_0^\prime]_L}{\lambda-\omega}\left(e^{(\lambda-\omega)t}-1\right)\right),$$
so that \begin{equation*}\begin{split}[\Psi_x(t)]_L&=
e^{\lambda(T-t)}[\phi_x(t)]_L\\
&\le e^{-(\lambda-2\omega)t}[\phi_0^\prime]_L+
\frac{[g_0^\prime]_L}{\lambda-\omega}(1-e^{-(\lambda-2\omega)t})\\
&\le [\phi_0^\prime]_L+ \frac{[g_0^\prime]_L}{\lambda-\omega},
\end{split}\end{equation*}
for all $t\ge0$.

Next we prove $(ii)$. Let $h$ be a real number $\vert h\vert\le1$
and $z\in\V$ such that $\vert z\vert_{\V}\le 1$.  Let $u_\e$ is
$\e$-optimal at $(0,0)$ with horizon $t$,
$y_{0,\e}(s):=y(s;0,0,u_\e)$ and $y_{h,\e}(s):=y(s;0,hz,u_\e)$,
then by means of (\ref{y-stima-1}) one has
$$\vert y_{0,\e}(t)\vertv\le \Vert B\Vert_{L(U,\V)}e^{\omega
t}\ro3(0,t)^{\frac{1}{q}}\Vert u_\e\Vert_{L^p(0,t;U)},$$ so that
\begin{equation}\begin{split}
&\frac{\Psi(t,hz)-\Psi(t,0)}{h}\le \\
&\ \ \ \ \ \ \ \le \int_0^t e^{-\lambda
s}[g_0(y_{h,\e}(s))-g_0(y_{0,\e}(s))]ds+e^{-\lambda t}[\phi_0(y_{h,\e}(t))-\phi_0(y_{0,\e}(t))]+\e\\
&\ \ \ \ \ \ \ \le \int_0^t e^{-\lambda s}\langle
g_0^\prime(y_{0,\e}(s)),e^{\omega s}\vert z\vertv\ranglev
ds+e^{\lambda t}\langle
\phi_0^\prime(y_{0,\e}(t)),e^{\omega t}\vert z\vertv\ranglev-\e\\
&\ \ \ \ \ \ \ \le  \vert z\vert_{\V}\left[\vert
g_0^\prime\vert_{B_1}\int_0^te^{-(\lambda-\omega) s}\left(1+ \vert
y_{0,\e}(s)\vert_{\V}\right)ds+e^{-(\lambda-\omega)
t}\vert\phi_0^\prime\vert_{B_1}(1+\vert
y_{0,\e}(t)\vert_{\V})\right]-\e\\
&\ \ \ \ \ \ \ \le \vert z\vert_{\V}\left[\vert
g_0^\prime\vert_{B_1}\left(\frac{1}{\lambda}+\Vert B\Vert\Vert
u_\e\Vert_{L^p_\lambda(0,t;U)}\int_0^t
e^{-(\lambda-2\omega) s}\ro3(0,s)^{\frac{1}{q}}ds\right)+ \right.\\
&\ \ \ \ \ \ \ \ \ \ \ \ \ \
+\left.\vert\phi_0^\prime\vert_{B_1}(e^{-(\lambda-\omega)
t}+e^{-(\lambda-2\omega) t}\ro3(0,t)^{\frac{1}{q}})
\right]-\e\\
\end{split}\end{equation}
Recalling  Lemma \ref{OC sono limitati}, using that $\lambda>q\omega$ and reasoning like in (\ref{omegapnot=lambda}) and
(\ref{rho31}), one obtains that the following quantities
$$\Vert u_\e\Vert_{L^p_\lambda(0,t;U)}, \ \ e^{-(\lambda-2\omega) t}\ro3(0,t)^{\frac{1}{q}}, \ \ \ \int_0^t
e^{-(\lambda-2\omega) s}\ro3(0,s)^{\frac{1}{q}}ds$$
 are
bounded by a constant (independent of $t$), by passing to limits
as $h\to0$  in the preceding inequality one derives
$$\sup_{t\ge0}\langle\Psi_x(t,0),z\rangle<+\infty.$$
On the other hand, by a similar reasoning, and with $u_{h,\e}$
$\e$-optimal at $(0,hz)$ with horizon $t$, there exists a positive
constant $\eta$ such that
\begin{equation}\begin{split}
&\frac{\Psi(t,0)-\Psi(t,hz)}{h}\le \eta\vert
z\vert_{\V}-\e\\
\end{split}\end{equation}
so that $$\sup_{t\ge0}\langle\Psi_x(t,0),-z\rangle<+\infty,$$ and
the proof is complete.
\end{proof}

\subsection{Proof of Theorem \ref{main}}
 We divide the long proof into several steps.
Let $ x\in \V$ be fixed, $\vert x\vertv\le r$, and let $0\le t_1<t_2$.

{\it Claim 1: Let $\e>0$ be fixed and let $u_\e\in
L^p_\lambda(0,t;U)$ be $\e$-optimal at starting point $(0,x)$ with
horizon $t$, and $y_\e(s):=y(s;0,x,u_\e)$ be the associated
trajectory. Then there exists a bounded continuous function
$\rho$, depending only from $r$,  with
$\lim_{t\to+\infty}\rho(t)=0,$ and such that:

$e^{-\lambda t}\vert y_\e(t)\vertv^2\le\rho(t)$, if Assumptions
\ref{asst2} [8.a] are satisfied;

$e^{-\lambda t}\vert y_\e(t)\vertv\le\rho(t)$, if Assumptions
\ref{asst2} [8.b] are satisfied.}

\noindent Indeed, applying  (\ref{y-stima-1,7}) and Lemma \ref{OC
sono limitati}, we derive
\begin{equation}\label{stimay}\begin{split}
e^{-\lambda t}\vert y_\e(t)\vert_{\V}&\le K_r e^{-(\lambda
-\omega)t}\left[1+\ro3(0,t)^{\frac{1}{q}}\right],\end{split}\end{equation}
with $K_r$ a suitable constant. Hence for a (possibly different)
constant $K_r>0$, we have
\begin{equation}\label{stimay2}\begin{split}e^{-\lambda t}\vert y_\e(t)\vertv^2&
\le K_r e^{-(\lambda -2\omega)t}(1+\ro3(0,t)^{\frac{1}{q}}+\ro3(0,t)^{\frac{2}{q}}).\\
\end{split}\end{equation}
By proceeding as in the proof of Lemma \ref{stima y}  one sees that the following
functions are infinitesimal as $t$ goes to $+\infty$
$$e^{-(\lambda -\omega)t}\ro3(0,t)^{\frac{1}{q}},\ \ \
e^{-(\lambda -2\omega)t}\ro3(0,t)^{\frac{1}{q}},$$ so that what
is left to show is
$$e^{-(\lambda -2\omega)t}\ro3(0,t)^{\frac{2}{q}},\ \ t\to0.$$
The property is straightforward in the case $\lambda=\omega p$, as
$$e^{-(\lambda -2\omega)t}\ro3(0,t)^{\frac{2}{q}}=e^{-(\lambda -2\omega)t}\vert t\vert^{\frac{2}{q}},$$
while for $\lambda<\omega p$ one has
$$e^{-(\lambda -2\omega)t}\ro3(0,t)^{\frac{2}{q}}\le Ce^{-(\lambda
-2\omega)t}.$$ Finally, if $\lambda>\omega p$,
$$e^{-(\lambda -2\omega)t}\ro3(0,t)^{\frac{2}{q}}\le Ce^{-(\lambda -2\omega)t}
e^{2\left(\frac{\lambda}{p}
-\omega\right)t}=Ce^{\frac{\lambda}{p}\left(2-p\right)t}$$ and
Claim 1 is proved.

{\it Claim 2: $\displaystyle \lim_{t_1\to+\infty,t_1<t_2}
\Psi(t_1,x)-\Psi(t_2,x)\le0$ }

Let $\varepsilon>0$ be arbitrarily fixed,  let $u_\varepsilon\in
L^p_\lambda(0,t_2;U)$ be such that
$$\Psi(t_2,x)\ge J_{t_2}(0,x,u_\varepsilon)-\epsilon$$ and let
$y_\e(\tau):=y(\tau;0,x,u_\e)$. Then
\begin{equation}\begin{split}
\Psi(t_2,x)&\ge \int_0^{t_1}e^{-\lambda
s}[g_0(y_\varepsilon(s))+h_0(u_\varepsilon(s))]ds+ e^{-\lambda
  t_1}J_{t_2-t_1}(0,y_\varepsilon(t_1),u_\varepsilon(\cdot+t_1))-\varepsilon\\
  &\ge \Psi(t_1,x)-e^{-\lambda
t_1}\phi_0(y_\varepsilon(t_1))+e^{-\lambda
t_1}\Psi(t_2-t_1,y_\varepsilon(t_1))-\varepsilon\end{split}\end{equation}
where we used  (\ref{funzio}). Consequently, when [8.a] holds
\begin{equation}\begin{split}\Psi(t_1,x)-\Psi(t_2,x)
&\le C e^{-\lambda t_1}(1+\vert
y_\varepsilon(t_1)\vertv^2)+\varepsilon
\end{split}\end{equation}while for [8.b]
\begin{equation}\begin{split}\Psi(t_1,x)-\Psi(t_2,x)
&\le C e^{-\lambda t_1}(1+\vert
y_\varepsilon(t_1)\vertv)+\varepsilon
\end{split}\end{equation}
for a suitable constant $C$, and the last implies Claim 2 as a
consequence of Claim 1, and Lemma \ref{stima psi}.

{\it Claim 3: $\displaystyle \lim_{t_1\to+\infty,t_1<t_2} \Psi(t_2,x)-\Psi(t_1,x)\le0$
}

If we choose $\varepsilon>0$ and $v_\varepsilon\in
L^p_\lambda(0,t_1;U)$ so that
$$\Psi(t_1,x)\ge J_{t_1}(0,x,v_\varepsilon)-\varepsilon,$$
and set $y_\varepsilon(s):=y(s,0,x,u_\varepsilon)$,
then
by the DDP contained in (\ref{PPDT})
 we obtain
\begin{equation}\begin{split}\Psi(t_2,x)-\Psi(t_1,x)&\le
e^{-\lambda t_1}\Psi(t_2-t_1,y_\varepsilon(t_1))-e^{-\lambda
t_1}\phi_0(y_\varepsilon(t_1))+\varepsilon\\
\end{split}\end{equation}
which leads as before to the conclusion. Since the constants
involved in the estimates are uniform in $x$, for $x$ varying in a
bounded subset of $\V$, we derive the convergence $\Psi(t,x)\to
\Psi_\infty(x)$, as $t\to +\infty$ is uniform on bounded subsets
of $\V$.

Next we discuss the convergence of gradients.

{\it Claim 4: $\Psi_\infty$ is
Frech\'et differentiable, with differential $\Psi^\prime_\infty$ and,
for every fixed $x\in \V$
$$\Psi_x(t,x)\to\Psi^\prime_\infty(x), \ {\rm weakly\ in\ }V,\  {\rm as\ \ }t\to+\infty.$$}
 Let $x$ be fixed in $\V$, $h$ a real parameter,
 with $h\in[-1,1]$, $y$ in $\V$ with $\vert y\vert_{\V}\le 1$, and
$\xi_t(h;x,y)\equiv\xi_t(h):=\Psi(t,x+hy)$. Then
$$\xi_t^\prime(h):=\langle\Psi_x(t,x+hy), y\rangle.$$
Note that, since $\xi_t(h)\to\xi_\infty(h)\equiv
\Psi_\infty(x+hy)$ as $t\to+\infty$, if we show that
$\xi_t^\prime(h)$ converges uniformly in $[-1,1]$ to some function
as
 $t\to+\infty$ (or along a subsequence),  then  such function is
 $\xi_\infty^\prime(h)$. We do so
  by means of Ascoli--Arzel\`a Theorem.
We have
$$\vert\xi_t^\prime(h)-\xi_t^\prime(k)\vert\le\vert
y\vert^2_{\V}[\Psi_x(t))]_L\vert h-k\vert_{\mathbb{R}}$$ which implies, by
Lemma \ref{stimegrad} $(i)$, that the family
$\{\xi_t^\prime\}_{t\ge0}$ is equicontinuous (more precisely,
equilipschitzean). Moreover
\begin{equation}\begin{split}
\vert\xi_t^\prime(h)\vert&\le \vert
y\vert_{\V}\vert\Psi_x(t,x+hy)\vert_V\\
&\le\vert y\vert_{\V}\left([\Psi_x(t)]_L\vert x+h
y\vert_{\V}+\vert\Psi_x(t,0)\vert_V\right)
\end{split}\end{equation} from which follows, by means of Lemma \ref{stimegrad}
$(ii)$, that $\{\xi_t^\prime\}_{t\ge0}$ is  uniformly bounded.
Consequently, $\Psi_\infty$ is Gateaux differentiable. Indeed,
there exists the following limit
$$\lim_{h\to0}\frac{\Psi_\infty(x+hy)-\Psi_\infty(x)} {h}=\xi_\infty^\prime(0)=: \langle
\nabla\Psi_\infty(x),y\ranglev,$$ where
 $\nabla\Psi_\infty$
 indicates the Gateaux differential of $\Psi_\infty$. In
 particular, what we prove implies
 also
$$\Psi_x(t,x)\to \nabla\Psi_\infty(x)~{\rm  weakly ~in~}V~{\rm  as}~ t\to+\infty.
$$ Finally we show that $\nabla\Psi_\infty$
is continuous. It suffices to pass to limits as $t$ goes to
$\infty$ in \begin{equation}\begin{split}
\vert\xi_t^\prime(0;x,y)-\xi_t^\prime(0;z,y)\vert&\le\vert
y\vert_{\V}\vert \Psi_x(t,x)-\Psi(t,z)\vert_V\\
&\le \vert y\vert_{\V}\sup_{t\ge0}[\Psi_x(t)]_L\vert
x-z\vertv.\end{split}\end{equation} Hence $\Psi_\infty$ is
Frech\'et differentiable with Frech\'et differential
$\Psi_\infty^\prime=\nabla\Psi_\infty$. The proof that
$\Psi_\infty$ is convex and in $C^1_{Lip}(\V)$ is trivial by means
of Lemma \ref{stimegrad}

\subsection{Proof of Theorem \ref{th:mainnew} $(i)$}

 First of all we show that, for any
fixed $t$, $x$ and $u\in L^p_\lambda(0,+\infty)$ we have
\begin{equation}\label{J_t-J_infty}\exists \lim_{t\to +\infty}J_t(0,x,u)
=J_\infty(0,x,u).\end{equation}
We separately show that, if $y(s)=y(s;0,x,u)$, then
\begin{equation}\label{limiti}
\lim_{t\to+\infty}\left\vert e^{-\lambda
t}\phi_0(y(t))\right\vert=0, \ \ \ {\rm and}\ \ \
J_\infty(t,x,u)=\lim_{t\to+\infty}\int_0^te^{-\lambda
\tau}[g_0(y(\tau))+h_0(u(\tau))]d\tau.\end{equation} Indeed
$$\vert y(t)\vertv\le Ke^{\omega
t}(1+\ro3(0,t)^{\frac{1}{q}})$$ where $K=\vert x\vertv\vee\Vert u\Vert_{L^p_\lambda(0,+\infty;U)}$. Hence
\begin{equation}\label{stimay2new} e^{-\lambda t}\vert y(t)\vertv^2\le
\rho(t),\ \ e^{-\lambda t}\vert y(t)\vertv\le \rho(t)\end{equation} where $\rho(t)$ denotes
some positive function with $\lim_{t\to+\infty}\rho(t)=0$ (obtained as in proof of Claim 1 of Theorem
\ref{main}). These inequalities combined with Assumptions
\ref{asst2} [8.a]  or [8.b] -- recall that $\phi_0$ and $g_0$ are
either sublinear or subquadratic -- give the first equality in
(\ref{limiti}), and by means of dominated convergence also
$$\int_0^t e^{-\lambda s}g_0(y(s))ds\to \int_0^{+\infty} e^{-\lambda s}g_0(y(s))ds,
\ \ as\ \ t\to +\infty.$$ To complete the proof of (\ref{limiti})
we just observe
$$\int_0^t e^{-\lambda s}h_0(u(s))ds\to \int_0^{+\infty} e^{-\lambda s}h_0(u(s))ds,
\ \ as\ \ t\to +\infty$$
by monotone convergence, for $h_0$ is bounded from below.

Then (\ref{J_t-J_infty}) is proved. As a consequence, by passing
to limits as $t$ tends to $+\infty$ in
$$\Psi(t,x)\le J_t(0,x,u),\ \ \ \forall u\in L^p_\lambda(0,+\infty;U)$$
we obtain
$$\Psi_\infty(x)\le J_\infty(0,x,u), \ \ \forall u\in L^p_\lambda(0,+\infty;U)$$
which implies
$$\Psi_\infty(x)\le Z_\infty(0,x).$$
We now need to show that the reverse inequality holds. Let $\varepsilon>0$ be fixed.
  From (\ref{fond}) and (\ref{fond3}) we derive
that if $u^*$ is $\varepsilon$-optimal at $(0,x)$ with horizon $+\infty$
\begin{equation}\label{ennio}J_\infty(0,x,u^*)\le
C(1+\vert x\vertv^2)+\varepsilon\end{equation} with $C$ a suitable
constant. Hence, we set
 \begin{equation} u_t(s)=\begin{cases}u_1(s)&s\in[0,t]\\
u_2(s)&s\in]t,+\infty[\end{cases}\end{equation} where $u_1\in
L^p_\lambda(0,t;U)$ is $\varepsilon$-optimal for $J_t$ at $(0,x)$,
and $u_2\in L^p_\lambda(t,+\infty;U)$ is $\varepsilon$-optimal for
$J_\infty$ at $(0,y_t(t))$, with $y_t(s):=y(s;0,x,u_t)$, and we
derive by means of (\ref{funzio}) and (\ref{ennio}) the following
chain of inequalities
\begin{equation}\begin{split}
Z_\infty(0,x)&\le J_\infty(0,x,u_t)\\
&=J_t(0,x,u_1)+e^{-\lambda t}\big[J_\infty(0, y_t(t), u_2)-\phi_0(y_t(t))\big]\\
&\le \Psi(t,x)+ Ce^{-\lambda t}(1+\vert y_t(t)\vertv+\vert y_t(t)\vertv^2)+2\varepsilon\\
&=: \Psi(t,x)+ \rho(t)+2\varepsilon
\end{split}\end{equation}
and with $C$ some suitable constant (possibly different from the one
mentioned above). Note that $\rho(t)\to 0$, as $t\to+\infty$ as one
derives from Claim 1 in the proof of Theorem \ref{main}. Hence, by
passing to limits as $t$ goes to $+\infty$, we derive
$$Z_\infty(0,x)\le \Psi_\infty(x),$$
 and the proof is complete.

\subsection{Proof of Theorem \ref{th:mainnew} $(ii)$}

To prove the theorem we make use of the Dynamic Programming
Principle (DPP from now on) contained in the following Lemma.
\begin{lemma}\label{DPP}
We have
$$\Psi_\infty(x)=\inf_{u\in L^p_\lambda(0,+\infty;U)}\left\{\int_0^t e^{-\lambda s}(g_0(y(s))+h_0(u(s)))ds
+e^{-\lambda t}\Psi_\infty(y(t))\right\},\ \ \forall t>0.$$
Moreover, given any $\varepsilon >0$, if $u_\varepsilon$ is such
that
$$
J_\infty(0,x,u_\varepsilon)<\Psi_\infty(x) + \varepsilon
$$
then also
$$
\int_0^t e^{-\lambda s}(g_0(y(s))+h_0(u(s)))ds +e^{-\lambda
t}\Psi_\infty(y(t))< \Psi_\infty(x) + \varepsilon
$$
\end{lemma}
The proof of this lemma is standard and we omit it.

We prove then Theorem \ref{th:mainnew} $(ii)$. Let $t>0$ be fixed.
Let also $u(s)\equiv \bar u\in dom(h_0)$,  and let $\bar
y(s):=y(s;0,x,\bar u)$. Then the DPP implies
\begin{equation}\label{ada}\frac{e^{-\lambda t}\Psi_\infty(\bar y(t))-\Psi_\infty(x)}{t}
\ge-\frac{1}{t}\int_0^te^{-\lambda s}[g_0(\bar y(s))+h_0(\bar u)] ds.\end{equation}
Since
$$\frac{\bar y(t)-\bar y(0)}{t}=\frac{e^{At}x-x}{t}+\frac{1}{t}\int_0^te^{A(t-s)}B\bar u ds
\to Ax+B\bar u,\ \  as \ \ t\to 0,$$ and $s\mapsto
g_0(y(s))e^{-\lambda s}\in C(0,t;U)$, then we may pass to limits
in (\ref{ada}) and obtain
$$-\lambda \Psi_\infty(x)+\langle \Psi^\prime_\infty(x),
Ax\rangle+\langle\Psi^\prime_\infty(x), B\bar u\rangle +h_0(\bar
u)+g_0(x)\ge0,$$ and take the infimum of both sides as $\bar u\in
dom(h_0)$ and derive
$$ -\lambda \Psi_\infty(x)+
\langle \Psi^\prime_\infty(x),
Ax\rangle-h_0^*(-B^*\Psi^\prime_\infty(x))+g(x)\ge0.$$ Next we
prove the reverse inequality. Let $\e>0$ and $t\in[0,1]$
arbitrarily fixed, and let $u_\e$ be an $\e t$- optimal control at
$(0,x)$ with horizon $+\infty$, and $y_\e$ be the associated
trajectory. Then by the DPP
$$e^{-\lambda t}\Psi_\infty(y_\e(t))-\Psi_\infty(x)+\int_0^t e^{-\lambda s}
[g_0(y_\e(s))+h_0(u_\e(s))]ds\le \e t, \ \forall t\in[0,1].$$
Since $\Psi_\infty$ is convex and differentiable, the preceding
implies
\begin{equation}\begin{split}\label{davide}e^{-\lambda t}\langle\Psi^\prime_\infty(x),
\frac{y_\e(t)-x}{t}\rangle &-\Psi_\infty(x)\frac{1-e^{-\lambda
t}}{t}+\\
&+\frac{1}{t}\int_0^t e^{-\lambda s}
[g_0(y_\e(s))+h_0(u_\e(s))]ds\le \e , \ \forall
t\in[0,1].\end{split}\end{equation} Now we show that
\begin{equation}\label{silvia}\frac{1}{t}\int_0^t e^{-\lambda s}
g_0(y_\e(s))ds=g_0(x)+\rho (t)\end{equation} where by $\rho(t)$,
we denote some real function not depending
on $u_\e$ such that $\rho(t)\to0$ as $t\to0$. Indeed the
assumptions on $g_0$ imply that
$$\vert g_0(x)-g_0(y)\vert\le C(\vert x\vertv,\vert y\vertv)\vert x-y\vertv$$ with
$$C(\alpha,\beta):=\big([g_0^\prime]_{L}+\vert g_0^\prime(0)\vert\big)
\big(1+\alpha\vee\beta).$$
 Moreover by Lemma \ref{OC sono
limitati} and by (\ref{y-stima-1,7}) we derive
\begin{equation}\vert y_\e(s)-x\vertv\le\vert
e^{sA}x-x\vertv+C(1+\vert x\vertv^2)e^{\omega
s}\ro3(0,s)^{\frac{1}{q}}\end{equation} for some constant $C$
(independent of $u_\e$ and $x$) and with $\ro3$ the function
defined in Lemma \ref{stima y}, which has as a consequence
$$\sup_{s\in [0,1]}\vert y_\e(s)\vertv<K(x)<+\infty,$$ with $K(x)$
not depending on $u_\e$. Hence for $t\in[0,1]$
\begin{equation}\begin{split}
&\frac{1}{t}\int_0^t \left\vert e^{-\lambda s}
g_0(y_\e(s))ds-g_0(x) \right\vert ds \le
\frac{1}{t}\int_0^t e^{-\lambda s} |
g_0(y_\e(s))-g_0(x)| ds+\rho(t)\\
&\ \ \ \ \ \le C(\vert
x\vertv,K(x)) \left[ C(1+\vert x\vertv^2)\frac{1}{t}\int_0^t
e^{-(\lambda-\omega)s}\ro3(0,s)^{\frac{1}{q}}ds +
\frac{1}{t}\int_0^t e^{-\lambda s}\vert e^{sA}x-x\vertv ds \right]
\\
&\  \ \ \  \ \ \ \ \ \ \ \ \ \ \ \ \
+\rho(t),
\end{split}\end{equation}
which implies (\ref{silvia}) by definition of $\ro3(0,s)$.

Observe now  that, as $x\in D(A)$, then
$$\frac{y_\e(t)-x}{t}=Ax+\rho (t)+\frac{1}{t}\int_0^t e^{(t-s)A}Bu_\e(s)ds.$$
Then, the last and (\ref{silvia}) imply  that (\ref{davide}) can
be written as
\begin{equation}\begin{split}\langle&\Psi^\prime_\infty(x),
Ax\rangle -\lambda \Psi_\infty(x)+g(x)+\\
&+\frac{1}{t}\int_0^t e^{-\lambda s}\big[\langle e^{-\lambda
(t-s)} B^* e^{(t-s)A^*}\Psi^\prime_\infty(x), u_\e(s)\rangle+
h_0(u_\e(s))\big]ds \le\e +\rho(t),\\
& \ \ \forall t\in[0,1].\end{split}\end{equation} We then get that
\begin{equation}\begin{split}
&\frac{1}{t}\int_0^t e^{-\lambda s}\big[\langle e^{-\lambda (t-s)}
B^* e^{(t-s)A^*}\Psi^\prime_\infty(x), u_\e(s)\rangle+
h_0(u_\e(s))\big]ds\ge \\
&\ge -\frac{1}{t}\int_0^t e^{-\lambda s}\sup_{u\in U}\big[\langle
-e^{-\lambda (t-s)} B^* e^{(t-s)A^*}\Psi^\prime_\infty(x),
u\rangle- h_0(u)\big]ds=\\
&=-\frac{1}{t}\int_0^t e^{-\lambda s}h_0^*\left(-e^{-\lambda
(t-s)} B^*
e^{(t-s)A^*}\Psi^\prime_\infty(x)\right)ds\\
&=-h_0^*(-B^*\Psi^\prime_\infty(x))+\rho(t),\end{split}\end{equation}
for $h_0^*\in C^1_{Lip}(U)$ by assumption. Hence
$$\langle\Psi^\prime_\infty(x),
Ax\rangle -\lambda
\Psi_\infty(x)+g(x)-h_0^*(-B^*\Psi^\prime_\infty(x))\le \e+\rho(t),
\ \ \forall t\in [0,1],$$ which implies the thesis by passing to
limits as $t\to0$.

\subsection{Verification Theorem}
We state and prove the following verification theorem:

\begin{theo}\label{th:verif}
Let Assumptions \ref{asst2} hold. Let $t\ge 0$ and $x\in \V$ be
fixed. Then
\begin{equation} \label{eq:lemmafond}
e^{-\lambda t}\Psi_{\infty}(x)= J_\infty(t,x,u) - \int_t^T
 e^{-\lambda s}\big[
 h_0^*(-B^*\Psi_\infty^\prime(y(s)))+(B^*\Psi_\infty^\prime(y(s))\; \vert\; u(s))_U+h_0(u(s))]\big]
ds.
\end{equation}
As a consequence, an admissible pair $(u,y)$ at $(t,x)$ is optimal
if and only if
$$
\sup_{u\in U}\left\{ \left(u|-B^*\Psi^\prime_{\infty}(y(s))
\right)_U - h_0 (u)
 \right\}=
\left(u(s)|-B^*\Psi^\prime_{\infty}(y(s))  \right)_U - h_0 (u(s))
$$
for a.e. $s\ge 0$, which is equivalent to
$$
u(s) = (h_0^*)'(-B^*\Psi^\prime_{\infty}(y(s))
$$
for a.e. $s\ge 0$.
\end{theo}

\begin{proof}

Let first $x\in D(A)$, $t>0$ be fixed. Let $u$ be any admissible
control at $(t,x)$ such that $J_T (t,x,u) <+\infty$ for every
$T>0$. (Note that an admissible control violating this condition
cannot be optimal for the infinite horizon problem, as  $J_T
(t,x,u) =+\infty$ for some $T>0$ implies $J_\infty (t,x,u)
=+\infty$.) Let $y$ be the associated trajectory. Then for a.e.
$s\in[t,+\infty[$ we may differentiate $e^{-\lambda
s}\Psi_{\infty}(y(s))$ as function of $s$ to obtain
\begin{equation}\begin{split}
\frac{d}{ds}e^{-\lambda s}\Psi_{\infty}(y(s))&=- \lambda
e^{-\lambda s}\Psi_{\infty}(y(s))+
 e^{-\lambda s}\langle
\Psi^\prime_{\infty}(y(s)),Ay(s)+Bu(s) \rangle \\
=h_0^*(-B^*\Psi_\infty^\prime(y(s)))&-g_0(y(s))+(
B^*\Psi^\prime_{\infty}(y(s))\;\vert\; u(s) )_U
 +e^{-\lambda s}h_0(u(s)) -e^{-\lambda
 s}h_0(u(s))\end{split}\end{equation}
where we used the fact that $\Psi_{\infty}$ solves the stationary
HJB equation, and  we added and subtracted the term $e^{-\lambda
s}h_0(u(s))$. Integrating such equation on $[t,T]$ we have
\begin{equation}\label{eq:lemmafondT}\begin{split}
e^{-\lambda T}\Psi_{\infty}(y(T))-e^{-\lambda
t}\Psi_{\infty}(x)&=\\
=\int_t^T e^{-\lambda s}\big[
h_0^*(-B^*\Psi_\infty^\prime(y(s)))&+(B^*\Psi_\infty^\prime(y(s))\;
\vert\; u(s))_U+h_0(u(s))]\big]
ds+\\
&-J_T(t,x,u)+e^{-\lambda T}\phi_0(y(T)).
\end{split}\end{equation}
Such relation holds for all admissible controls.  If now we show
that, for any fixed admissible control $u$, is
$$e^{-\lambda T}\Psi_{\infty}(y(T))\to 0, \ and\ e^{-\lambda
T}\phi_0(y(T))\to 0, \ as \ T\to +\infty,$$ then we may pass to
limits in (\ref{eq:lemmafondT}) and derive (\ref{eq:lemmafond}).
Indeed, using the fact that $\phi_0$ and $\Psi_\infty$ are either
sublinear or subquadratic, we observe that, by (\ref{stimay2new})
and (\ref{stimaynew})
$$e^{-\lambda T}\vert \Psi_{\infty}(y(T))+\phi_0(y(T))\vert\le C
K\rho(T),$$ where $K$ and $\rho$ are the same there
considered, and $C$ a suitable positive constant. Then $(\ref{eq:lemmafond})$ is proved. It also implies
that $\Psi_{\infty}(x)\le J_\infty(0,x,u)$ and that the equality
holds if and only if
$$
\int_t^{+\infty}
 e^{-\lambda s}\big[
 h_0^*(-B^*\Psi_\infty^\prime(y(s)))+
 (B^*\Psi_\infty^\prime(y(s))\; \vert\; u(s))_U+h_0(u(s))]\big]ds=0
$$
which means, by the positivity of the integrand that
$$
 h_0^*(-B^*\Psi_\infty^\prime(y(s)))=
 (-B^*\Psi_\infty^\prime(y(s))\; \vert\; u(s))_U-h_0(u(s))
$$
for almost every $s\ge t$. The claim easily follows from the
definition of $h_0^*$. The claim for generic $x\in \V$ follows by
approximating it by a sequence of elements of $D(A)$ and observing
that the relation (\ref{eq:lemmafond}) make sense also for $x \in
\V$.\end{proof}

\subsection{Proof of Theorem \ref{th:uniquefeedback}}

We first observe that the closed loop equation has a unique solution
$y^*$ since the feedback map defined as
$$G(x)=(h_0^*)^\prime(-B^*\Psi_\infty(x))$$
is Lipschitz continuous, as one may show by a standard fixed point
argument.

Next we prove that the control
$$u^*(s) = (h_0^*)'(-B^*\Psi_\infty^\prime(y^*(s))$$
is admissible, i.e. it belongs to $L_\lambda^p(t,+\infty;U)$. To
do so, we first observe that the relation (\ref{eq:lemmafondT})
holds true for any control $u\in L^1_{loc}(t,+\infty;U)$ such that
$J_T (t,x,u) <+\infty$ for every $T>0$. Since by definition we
have
$$
h_0(u^*(s))=\left(-B^*\Psi_\infty^\prime(y^*(s))|u^*(s)\right)_U -
h_0^*\left(-B^*\Psi_\infty^\prime(y^*(s))\right)
$$
then also $J_T (t,x,u^*) <+\infty$ for every $T>0$. So we get that
$u^*$ satisfies
\begin{equation}\label{questa}
J_T(t,x,u^*)-e^{-\lambda T}\phi_0(y^*(T))=e^{-\lambda
t}\Psi_{\infty}(x)-e^{-\lambda
T}\Psi_{\infty}(y^*(T)).\end{equation}
By means of (\ref{y-stima-1}) and proceeding as in the proof of Lemma \ref{stima y},
one derives
$$\int_t^Te^{-\lambda s}\vert y^*(s)\vertv ds\le\gamma_1+\gamma_2\Vert u^*\Vert_{L_\lambda^p(t,T;U)}$$
where $\gamma_1,\gamma_2$ are suitable constants (depending on $x$, $t$), so that
\begin{equation}\label{equesta}
\begin{split}
A:&=J_T(t,x,u^*)- e^{-\lambda T}\phi_0(y^*(T))\\ &\ge \int_t^T e^{-\lambda
s}[a\vert u^*(s)\vert_U+b]ds-C\int_t^T e^{-\lambda s}[1+\vert
y^*(s)\vertv] ds\\
&\ge a\Vert
u^* \Vert^p_{L_\lambda^p(t,T;U)}+\frac{b-C}{\lambda}-\gamma_1-\gamma_2\Vert u^*\Vert_{L_\lambda^p(t,T;U)}\\
&\ge a\Vert
u^* \Vert^p_{L_\lambda^p(t,T;U)}-\gamma_2\Vert
u^* \Vert_{L_\lambda^p(t,T;U)}-\gamma_3
\end{split}\end{equation}
for a suitable constant $\gamma_3$.
On the other hand, since by means of (\ref{y-stima-1}), and proceeding again like in the proof of Lemma \ref{stima y}, we have
$$e^{-\lambda T}(1+\vert y^*(T)\vertv)\le\gamma_4+\gamma_5\Vert u^*\Vert_{L_\lambda^p(t,T;U)}$$
where $\gamma_4,\gamma_5$ are suitable constants (depending on $x$, $t$), so that, using Lemma \ref{stima psi}
\begin{equation}\label{eanchequesta}
\begin{split}
B:=&e^{-\lambda
t}\Psi_{\infty}(x)-e^{-\lambda
T}\Psi_{\infty}(y^*(T))\\
&\le e^{-\lambda
t}\Psi_{\infty}(x)-C(\gamma_4+\gamma_5\Vert u^*\Vert_{L_\lambda^p(t,T;U)})
\end{split}\end{equation}
Hence combining this inequality with (\ref{equesta}) and (\ref{eanchequesta}) we obtain
\begin{equation*}\begin{split}a\Vert
u^*\Vert^p_{L^p_\lambda(t,T;U)}-\gamma\Vert u^*\Vert_{L^p_\lambda(t,T;U)}\le\eta
 \end{split}\end{equation*} for suitable positive constants $\gamma$ and $\eta$ independent of $T$.
 Then we may pass to
limits as $T$ goes to $+\infty$ and derive
$$a\Vert
u^*\Vert^p_{L^p_\lambda(t,+\infty;U)}-\gamma\Vert u^*\Vert_{L^p_\lambda(t,+\infty;U)}\le\eta,$$ which implies that
$u^*\in L^p_\lambda(t,+\infty;U)$, and is hence admissible.

Then by the above Theorem \ref{th:verif} we get that the couple
$(u^*,y^*)$ is optimal. The uniqueness follows by the uniqueness of
the solution of the closed loop equation.

\subsection{Proof of Thorem \ref{th:mainnew} $(iii)$}

Let $\psi \in C^1$ be any other function that satisfies the
stationary HJB equation  (\ref{SHJB}) in classical sense as from
Definition \ref{defsolSHJB}. Then we have, arguing exactly as in
the proof of Theorem \ref{th:verif}, that $\psi$ satisfies the
fundamental relation (\ref{eq:lemmafond}) in place of
$\Psi_\infty$. This implies, setting $t=0$,  that $\psi(x)\le
J_\infty(0,x,u)$ and consequently
$$\psi (x) \le \Psi_\infty (x).$$ Moreover if we find an
admissible control $u$ at $x$ such that
$$(h_0^*)^\prime(-B^*\psi^\prime(y(s)))=(-B^*\psi^\prime(y(s))\; \vert\;
u(s))_U-h_0(u(s))$$ then $\psi(x)= J_\infty(0,x,u)$,  $u$ is
optimal, and then $\psi (x) = \Psi_\infty (x)$. Such a control
exists as one may derive arguing as in the proof of the Theorem
\ref{th:uniquefeedback}.

\section{The economic example of optimal investment with vintage capital}\label{es eco}

We here describe our motivating example: the infinite horizon
problem of optimal investment with vintage capital, in the setting
introduced by Barucci and Gozzi \cite{BG1}\cite{BG2}, and later
reprised and generalized by Feichtinger et al. \cite{F1,F2,F3}, and by
  Faggian \cite{Fa2,Fa3}.

The capital accumulation is described by the following system
\begin{equation}\label{ipde}\begin{cases}&\frac{\partial y( \tau
,s)}{\partial \tau }+
 \frac{\partial y( \tau ,s)}{\partial s}+
\mu y( \tau ,s) =u_1( \tau ,s), \quad (\tau,s)
\in]t,+\infty[\times
]0,\bar s]\\
& y( \tau ,0)=u_0( \tau ), \quad \tau \in]t,+\infty[\\
&y(t,s)=x(s), \quad s\in [0,\bar s]\end{cases}\end{equation} with
$t>0$ the initial time, $\bar s\in[0,+\infty]$ the maximal allowed
age, and $\tau\in[0,T[$ with  horizon $T=+\infty$. The unknown
$y(\tau,s)$ represents the amount of capital goods of age $s$
accumulated at time $\tau$, the initial datum is a function $x\in
L^2(0,\bar s)$, $\mu>0$ is a depreciation factor. Moreover,
$u_0:[t,+\infty[\to {{\mathbb R}}$ is the investment in new
capital goods ($u_0$ is  the boundary  control) while
$u_1:[t,+\infty[\times[0,\bar s]\to {{\mathbb R}}$ is the
investment at time $\tau$ in capital goods of age $s$ (hence,
the distributed control). Investments are jointly referred to as
the control $u=(u_0,u_1)$.

Besides, we consider the firm profits represented by the
functional
$$I(t,x;u_0,u_1 )=\int_t^{+\infty}e^{-\lambda\tau} [R({Q(\tau)})
-c({u(\tau)})]d\tau$$ where, for some
given measurable coefficient $\a$, we have that
$${Q(\tau)}=\int_0^{\bar
s}\alpha(s){{y(\tau,s)}}ds$$ is the output rate (linear in
$y(\tau)$) $R$ is a concave revenue from $Q(\tau)$ (i.e., from
$y(\tau)$). Moreover we have
$$c({u_0(\tau),u_1(\tau)})=\int_0^{\bar{s}}
  c_1(s,u_1(\tau,s))ds +c_0(u_0(\tau)),$$
  with $c_1$ indicating the investment cost rate for
  technologies of age $s$, $c_0$  the investment cost in new
  technologies, including adjustment-innovation, $c_0$, $c_1$ convex in the control variables.

The entrepreneur's problem is that of maximizing $I( t, x;u_0,u_1
\kern-1pt)$ over all state--control pairs $\{y, (\kern-1pt
u_0,u_1\kern-1pt)\kern-1pt\}$ which are solutions in a suitable
sense of equation (\ref{ipde}).
Such problems are known as {\em vintage capital} problems, for the capital goods depend jointly on time $\tau$ and on age $s$, which is equivalent to their dependence from time and  vintage $\tau-s$.

\bigskip
When rephrased in an infinite dimensional setting, with
$H:=L^2(0,\bar s)$ as state space, the state equation (\ref{ipde})
can be reformulated as a linear control system with an unbounded
control operator, that is

\begin{equation}\begin{cases}\label{eq:statoecinH}
y^{\prime }(\tau)=A_0 y(\tau)+Bu(\tau), &\tau\in]t,+\infty[;\\
y(t)=x,\end{cases}
\end{equation}
where $y:[t,+\infty[\to H$, $x\in H$, $A_0:D(A_0)\subset H\to H$ is
the infinitesimal generator of a strongly continuous semigroup $\{
e^{\A t}\}_{t\ge0}$ on $H$ with domain $D(A_0)=\{f\in H^1(0,\bar
s):f(0)=0\}$ and defined as $A_0 f(s)=-f^\prime(s)-\mu f(s)$, the
control space is $U={{\mathbb R}}\times H$, the control function is
a couple $u\equiv(u_0,u_1):[t,+\infty[\to {{\mathbb R}}\times H$,
and the control operator is given by $Bu\equiv B(u_0,u_1)=
u_1+u_0\delta _{0}$, for all $(u_0,u_1)\in{{\mathbb R}}\times H$,
$\delta_0$ being the Dirac delta at the point $0$. Note that,
although $B\not\in L(U,H)$, is $B\in L(U,D(\A^*)^\prime)$. The reader can find
in \cite{BG1}the (simple)
proof of the following theorem, which we will exploit in a short while.
\begin{theo} \label{th:statoinH} Given any initial datum $x\in H$ and control
$u \in L^p_\lambda (t,+\infty;U)$ the mild solution of the equation
(\ref{eq:statoecinH})
$$
y(s)=e^{(s-t)A}x+\int_t^s e^{(s-\tau)A}Bu(\tau)d\tau
$$
belongs to $C([t,+\infty);H)$.
\end{theo}

Following Remark \ref{noidentif}, we then set $$V=D(A_0^*)=\{f\in
H^1(0,\bar s):f(\bar s)=0\}$$ and $\V =D(\A^*)^\prime$.
Regarding the target functional, we set
$$J_\infty(t,x;u):=-I(t,x;u_0,u_1),$$
with:

 $g_0:\V\to
{{\mathbb R}}, ~g_0(x) =- R( \langle\alpha , x\rangle),$


$h_0: U\to {{\mathbb R}}, ~h_0(u)=c_0(u_0)+\int_0^{\bar{s}} c_1(s,
u_1(s)) ds.$

\begin{rem} As announced in Remark \ref{estensioni}, here the extension of the datum
 $g_0$ to $\V$
is straightforward, as long as we assume that $\alpha\in V$ and replace scalar product in $H$ with the duality
in $V,\V$.

Note further that $\omega=0$, $\lambda>0$ (the type of the semigroup
is negative and equal to $-\mu$).  \hfill\qed\end{rem}

As the problem now fits into our abstract setting, the main results of the previous sections apply to the economic problem when data
$R$, $c_0$, $c_1$ satisfy Assumption \ref{asst2}[8.a] or [8.b]. In
particular, such thing happens in the following two interesting cases:

\begin{itemize}
    \item     If we assume, for instance, that
$R$ is a concave, $C^1$, sublinear function (for example one could
take $R$ quadratic in a bounded set and then take its linear
continuation, see e.g. \cite{F1,F3}), and $c_0$, $c_1$ quadratic
functions of the control variable, then Assumption \ref{asst2}[8.b]
holds.
    \item Assumption \ref{asst2}[8.a] is instead satisfied when $R$ is, for
instance, quadratic - as it occurs in some other meaningful economic
problems - and $c_0$, $c_1$ are equal to $+\infty$ outside some
compact interval, and equal to any convex {\it l.s.c.} function
otherwise. Such case corresponds to that of constrained controls
(controls that violate the constrain yield infinite costs).
\end{itemize}


In these cases, Theorems \ref{main}, \ref{th:mainnew},
\ref{th:uniquefeedback} hold true. In particular Theorem
\ref{th:uniquefeedback} states the existence of a unique optimal
pair $(u^*,y^*)$ for any initial datum $x \in \V$. Note that in
general the optimal trajectory $y^*$ lives in $\V$. However, since
the economic problem makes sense in $H$, we would now like to
infer that  whenever   $x$ (the initial age distribution of
capital) lies in $H$, then   the whole optimal trajectory lives in
$H$. Indeed, this is guaranteed by Theorem \ref{th:statoinH}.

All these results allow to perform the analysis of the behavior of
the optimal pairs and to study phenomena such as the diffusion of
new technologies (see e.g. \cite{BG1,BG2}) and the anticipation
effects (see e.g. \cite{F1,F3}). With respect to the results in
\cite{BG1,BG2}, here also the case of nonlinear $R$ (which is
particularly interesting from the economic point of view, as it
takes into account the case of large investors) is considered.
With respect to the results in \cite{F1,F3}, here the existence of
optimal feedbacks yields a tool to study more deeply the long run
behavior of the trajectories, like the presence of long run
equilibrium points and their properties.

\end{document}